 \documentclass[12pt,a4paper]{amsart}
\pdfoutput=1

\usepackage{amsmath,amssymb,amsthm}  
\usepackage{mathtools}
\usepackage{tikz}
\usepackage{thmtools, thm-restate}
\usepackage{float}
\usepackage{enumitem}

\usepackage{xcolor} 	
\usepackage{hyperref}
\hypersetup{
	colorlinks,
    linkcolor={red!60!black},
    citecolor={green!60!black},
    urlcolor={blue!60!black},
}
\usepackage[abbrev, msc-links]{amsrefs} 
\usepackage{comment}

\usepackage[utf8]{inputenc}
\usepackage[T1]{fontenc}
\usepackage{lmodern}
\usepackage[babel]{microtype}
\usepackage[english]{babel}

\usepackage{geometry}
\usepackage{enumitem}

\theoremstyle{plain}
\newtheorem{thm}{Theorem}[section]

\newtheorem{prop}[thm]{Proposition}

\newtheorem{claim}[thm]{Claim}
\newtheorem{cor}[thm]{Corollary}
\newtheorem{lem}[thm]{Lemma}

\newtheorem{obs}[thm]{Observation}

\newtheorem{quest}[thm]{Question}
\newtheorem{conj}[thm]{Conjecture}

\theoremstyle{definition}

\title{On the Packing/Covering Conjecture of Infinite Matroids}

\author{Attila Jo\'{o}}
\thanks{The author would like to thank the generous support of the Alexander 
von Humboldt Foundation and NKFIH 
OTKA-129211}
\address{Attila Jo\'{o},
University of Hamburg, Department of Mathematics, Bundesstra{\ss}e 55 (Geomatikum), 20146 Hamburg, Germany}
\email{attila.joo@uni-hamburg.de}
\address{Attila Jo\'{o},
Alfr\'{e}d R\'{e}nyi Institute of Mathematics, Set theory and general topology research division, 13-15 Re\'{a}ltanoda St., 
Budapest, Hungary}
\email{jooattila@renyi.hu}

\begin{document}

\begin{abstract}
The Packing/Covering Conjecture was introduced by  Bowler and  Carmesin  motivated by the Matroid Partition Theorem 
by Edmonds and Fulkerson. A packing for a family $ (M_i: i\in\Theta) $ of matroids on the common 
edge set $ E $ is a system $ (S_i: i\in\Theta ) $ of pairwise disjoint subsets of $ E $ where $ S_i $ is panning in $ M_i $. Similarly, 
a covering is a system $ (I_i: i\in\Theta ) $ with  $\bigcup_{i\in\Theta} I_i=E $ where $ I_i $ is independent in $ M_i $. The 
conjecture states that for every matroid family on $ E $ there is a partition $E=E_p \sqcup E_c$ such that 
$ (M_i \upharpoonright E_p: i\in \Theta) $ admits a packing and $ (M_i.  E_c: i\in \Theta) $ admits a covering. We prove the 
special case where $ E $ is countable and each $ M_i $ is either finitary or cofinitary. The connection between packing/covering 
and  matroid intersection problems discovered by Bowler and Carmesin can 
be established for every well-behaved matroid class. This makes possible to approach the problem from the direction of matroid 
intersection. We show that the generalized version of   Nash-Williams' Matroid Intersection Conjecture  holds for  
countable matroids having only finitary and cofinitary 
components.  
\end{abstract}

\maketitle
\section{Introduction}
Rado asked in 1966 (see Problem P531 in \cite{rado1966abstract}) if it is possible to extend the concept of matroids to infinite 
without loosing  
duality and minors. Based on the works of Higgs (see \cite{higgs1969matroids}) and  
Oxley  (see \cite{oxley1978infinite} and \cite{oxley1992infinite}) Bruhn, Diestel, Kriesell, Pendavingh and Wollan  
settled Rado's problem affirmatively in \cite{bruhn2013axioms} by finding a
set of cryptomorphic axioms for infinite matroids, generalising the usual independent set-, bases-, circuit-, closure- and 
rank-axioms for finite mastoids. Higgs named originally these structures B-matroids to distinguish from the original concept. 
Later this terminology vanished and in the context of infinite combinatorics B-matroids are referred as matroids  and the term 
`finite matroid' is used to 
differentiate.

An $ M=(E, \mathcal{I}) $ is a matroid  if $ \mathcal{I}\subseteq \mathcal{P}(E) $ with
\begin{enumerate}
[label=(\arabic*)]
\item\label{item axiom1} $ \varnothing\in  \mathcal{I} $;
\item\label{item axiom2} $ \mathcal{I} $ is downward closed;
\item\label{item axiom3} For every $ I,J\in \mathcal{I} $ where  $J $ is $ \subseteq $-maximal in $ \mathcal{I} $ but $ I $ is 
not, there exists an $  e\in J\setminus I $ such that
$ I+e\in \mathcal{I} $;
\item\label{item axiom4} For every $ X\subseteq E $, any $ I\in \mathcal{I}\cap 
\mathcal{P}(X)  $ can be extended to a $ \subseteq $-maximal element of 
$ \mathcal{I}\cap \mathcal{P}(X) $.
\end{enumerate}

If $ E $ is finite, then \ref{item axiom4} is redundant and \ref{item axiom1}-\ref{item axiom3} is one of the usual 
axiomatizations 
of finite matroids. One can show that every dependent set in an infinite matroid contains a minimal dependent set which is called 
a circuit. Before Rado's program 
was settled, a more restrictive axiom was used as a replacement of \ref{item axiom4}:

\begin{enumerate}[label={(4')}]
\item\label{item axiom4'} If all the finite subsets of an $ X\subseteq E $ are in $ \mathcal{I} $, then $ X\in \mathcal{I} $.
\end{enumerate}

The implication \ref{item axiom4'}$ \Longrightarrow $\ref{item axiom4} follows directly from Zorn's lemma thus  axioms 
\ref{item axiom1}, \ref{item axiom2}, \ref{item axiom3} and \ref{item axiom4'}  describe a subclass  $ \mathfrak{F} $ of the 
the matroids. This $ \mathfrak{F} $ consists of the matroids having  only finite circuits and called the class of \emph{finitary} 
matroids. Class $ \mathfrak{F} $ is closed under several important operations like direct sums and taking minors but not under 
taking 
duals which was the main motivation of Rado's program for looking for a more general matroid concept. The class $ 
\mathfrak{F}^{*} $ of the duals of the matroids in $ \mathfrak{F} $ 
consists of the \emph{cofinitary} matroids, i.e. matroid whose all cocircuits are finite. In order to being closed under all  
matroid operations we need, we work with the class $ 
\mathfrak{F}\oplus \mathfrak{F}^{*} $  of matroids having only finitary and cofinitary components, equivalently that are the 
direct sum of a finitary and cofinitary matroid.

Matroid union is a fundamental concept in the theory of finite matroids. For a finite family $ (M_i: i\leq n) $ of matroids on a 
common finite edge set $ E $ one can define a matroid $ \bigvee_{i\leq n}M_i $ on $ E $ by letting $ I\subseteq E $ be 
independent in $ \bigvee_{i\leq n}M_i $ if  $ I=\bigcup_{i\leq n}I_i $ for suitable $ I_i\in \mathcal{I}_{M_i} $ (see 
\cite{edmonds1968matroid}). This 
phenomenon fails for  infinite families of finitary matroids. Indeed, let $ E $ be uncountable and let $ M_i $ be the $ 1 
$-uniform 
matroid on $ E $ for $ i\in \mathbb{N} $. Then exactly the countable subsets of $ E $ would be independent in 
$ \bigvee_{i\in \mathbb{N}}M_i $ and hence there would be no maximal independent set contradicting 
\ref{item axiom4}.\footnote{For a finite family of finitary matroids the union operation  results in a matroid (Proposition 4.1 
in \cite{aigner2018intersection}).} Even so, 
Bowler and Carmesin observed (see section 3 in \cite{bowler2015matroid}) that the rank formula in the Matroid Partition 
Theorem by Edmonds and Fulkerson (Theorem 13.3.1 in \cite{frank2011connections}), namely:
\[ r\left( \bigvee_{i\leq n}M_i \right) =\max_{I_i\in \mathcal{I}_{M_i}}\left|\bigcup_{i\leq n}I_i\right|=\min_{E=E_p\sqcup 
E_c}\left|E_c\right|+\sum_{i\leq n}r_{M_i}(E_p),\] 
can be interpreted in infinite setting via the complementary slackness conditions. In the minimax formula above there is equality 
for the family $ (I_i: i\leq n) $  and partition $ E=E_p\sqcup E_c $ iff  
\begin{itemize}
\item $ I_i $ is independent in $ M_i $,
\item $\bigcup_{i\leq n} I_i\supseteq E_c $,
\item $ I_i\cap E_p $ spans $ E_p $ in $ M_i $ for every $ i $,
\item  $ I_i\cap I_j\cap E_p=\varnothing $ for $ i\neq j$.
\end{itemize}
 Bowler and Carmesin conjectured that for 
every family $\mathcal{M}:= (M_i: i\in\Theta) $  of matroids on a common edge set $ E $ there is a family $ (I_i: i\in \Theta) $  
and partition $ 
E=E_p\sqcup E_c $ satisfying 
the conditions above. To explain the name ``Packing/Covering Conjecture'' let us provide an alternative formulation.  A 
\emph{packing} for $ \mathcal{M} $ is a system $ (S_i: i\in\Theta ) $ of pairwise disjoint subsets of $ E $ where $ S_i $ is 
spanning in $ M_i $. Similarly, a \emph{covering} for $ \mathcal{M} $  is a  system $ (I_i: i\in\Theta ) $ with  
$\bigcup_{i\in\Theta} I_i=E $ where $ I_i $ is independent in $ M_i $. 

\begin{conj}[Packing/Covering, Conjecture  1.3 in \cite{bowler2015matroid} ]\label{conj: Pack/Cov}
For every family $ (M_i: i\in \Theta) $ of matroids on a common edge set $ E $  there is a partition $E=E_p \sqcup E_c$ in such a 
way that $ (M_i \upharpoonright E_p: i\in \Theta) $ admits a packing and $ (M_i.  E_c: i\in \Theta) $ admits a covering. 
\end{conj}

We shall prove the following special case of the Pacing/Covering Conjecture \ref{conj: Pack/Cov}: 

\begin{restatable}{thm}{PC}\label{t: main result0}
  For every family $ (M_i: i\in \Theta) $ of matroids on a common countable edge set $ E $ where $ M_i \in 
  \mathfrak{F}\oplus \mathfrak{F}^{*} $,  there is a partition $E=E_p \sqcup E_c$ such 
  that $ (M_i \upharpoonright E_p: i\in \Theta) $ admits a packing and $ (M_i.  E_c: i\in \Theta) $ admits a covering.   
\end{restatable}

It is worth to mention that packings and coverings have a crucial role in other problems as well. For example if $ (M_i: i\in\Theta) 
$ is as 
in Theorem 
\ref{t: main result0} and admits 
both a packing and a covering, then there is a partition 
$ E=\bigsqcup_{i\in \Theta}B_i $ where $ B_i $ is a base of $ M_i $ (see \cite{erde2019base}). Maybe surprisingly, the failure 
of the analogous statement 
for arbitrary 
matroids  is consistent with set theory ZFC (Theorem 1.5 of \cite{erde2019base}) which might raise some scepticism about the 
provability of  Conjecture 
\ref{conj: Pack/Cov} for general matroids.

The Packing/Covering Conjecture \ref{conj: Pack/Cov} is closely related to the Matroid Intersection Conjecture which has 
been one of the central open problems in the theory of infinite matroids: 

\begin{conj}[Matroid Intersection Conjecture by Nash-Williams, \cite{aharoni1998intersection}]\label{MIC}
If $ M $ and $ N $ are finitary matroids on the same edge set $ E $, then they admit a common 
 independent set $ I $ for which there is a partition $ E=E_M\sqcup E_N $ such that $ I_M:=I\cap E_M $ spans $ E_M $ in $ 
 M $ and  $ I_N:=I\cap E_N $ spans $ E_N $ in $ N $.  
\end{conj}

 Aharoni proved in \cite{aharoni1984konig} based on his earlier works with  Nash-Williams and  Shelah (see 
\cite{aharoni1983general} and \cite{aharoni1984another}) 
that the special case of Conjecture \ref{MIC} where $ M $ and $ N $ are partition matroids holds. The conjecture 
is also known 
to be true  if we assume that $ E $ is countable but $ M $ and 
$ N $ can be otherwise arbitrary (see \cite{joo2020MIC}).
Let us call Generalized Matroid Intersection Conjecture what we obtain from \ref{MIC} by extending it to arbitrary matroids 
(i.e. omitting the word ``finitary''). Several partial results has been obtained for this generalization but only for well-behaved 
matroid classes. The positive answer is known for example if: $ M $ is finitary and $ N $ is cofinatory 
\cite{aigner2018intersection}  or both 
matroids are 
singular\footnote{A matroid is singular if it is the direct sum of $ 1 $-uniform matroids and duals of $ 1 
$-uniform matroids.} and countable \cite{ghaderi2017}  or $ M $ is arbitrary and $ N $ is the direct sum of 
finitely many uniform matroids \cite{joó2020intersection}. 

Bowler and Carmesin showed that their Pacing/Covering Conjecture \ref{conj: Pack/Cov} and the 
Generalized Matroid Intersection Conjecture are equivalent and they also found  an important reduction for both (see Corollary 
3.9 in \cite{bowler2015matroid}). By analysing their proof it is clear that the equivalence can be established if we restrict both 
conjectures to a class of matroids closed under certain operations. It allows us to prove  Theorem \ref{t: main result} by showing 
the following instance of the  Generalized Matroid Intersection Conjecture which itself is a common extension of the singular case 
by Ghaderi \cite{ghaderi2017} and our previous work \cite{joo2020MIC}:

\begin{restatable}{thm}{MI}\label{t: main result}
  If $ M $ and $ N $ are matroids in $  \mathfrak{F}\oplus \mathfrak{F}^{*} $ on the same countable edge set $ E $, then they 
  admit a common independent set $ I $ for which there is a partition $ E=E_M\sqcup E_N $ such that $ I_M:=I\cap E_M $ spans 
  $ E_M $ in $ M $ and  $ I_N:=I\cap E_N $ spans $ E_N $ in $ N $.   
\end{restatable}

The paper is organized as follows. In the following section we introduce some notation and fundamental facts about matroids that 
are mostly well-know for finite ones. In Section \ref{s: premil} we collect some previous results and relatively easy technical 
lemmas in order be able the discuss  later the proof of the main results  without any distraction. Then  in Section 
\ref{s: reduction} we reduce the main results to a key-lemma. After these preparations the actual proof begins with 
Section \ref{s: augP} by developing and analysing an `augmenting path' type of technique.  Our main principle from this point is 
to 
handle the 
finitary and the 
cofinitary parts of matroid $ N $ differently in order to exploit the advantage 
of the finiteness of the  circuits and cocircuits respectively.  Equipped with these ``mixed'' augmenting paths we discuss the proof 
of our key-lemma in   Section 
\ref{s: proof of key-lemma}. Finally, we introduce an application in Section \ref{s: application} about orientations of a graph with 
in-degree requirements.

\section{Notation and basic facts}\label{sec notation}

In this section we introduce the notation and recall some basic facts about 
matroids that  we will use later without further explanation. For more details we refer to  \cite{nathanhabil}.

An enumeration of a countable set $ X $ is an $ \mathbb{N}\rightarrow X $ surjection that we write as $ \{x_n: n\in \mathbb{N} 
\} $.  We denote the symmetric difference $ (X\setminus Y)\cup (Y\setminus X) $ of $ X $ and $ Y $ by $ 
\boldsymbol{X\vartriangle Y} $.    
A pair ${M=(E,\mathcal{I})}$ is a \emph{matroid} if ${\mathcal{I} \subseteq \mathcal{P}(E)}$ satisfies  the axioms 
\ref{item axiom1}-\ref{item axiom4}.
The sets in~$\mathcal{I}$ are called \emph{independent} while the sets in ${\mathcal{P}(E) \setminus \mathcal{I}}$ are 
\emph{dependent}. An $ e\in E $ is a \emph{loop} if $ \{ e \} $ is dependent.
If~$E$ is finite, then  \ref{item axiom1}-\ref{item axiom3}  is one of the the usual axiomization of matroids in terms of 
independent sets  (while \ref{item axiom4} is redundant).
The maximal independent sets are called \emph{bases}. If $ M $ admits a finite base, then all the bases have the same size which 
is the rank $ \boldsymbol{r(M)} $ of $ M $ otherwise we let $ r(M):=\infty $.\footnote{It is independent of ZFC that the bases of 
a fixed matroid  have the same 
size (see \cite{higgs1969equicardinality} and \cite{bowler2016self}).} The minimal dependent sets are called \emph{circuits}. 
Every dependent set contains a circuit. The  \emph{components} of a matroid are the 
components of the hypergraph of its circuits.  The 
\emph{dual} of a matroid~${M}$ is the 
matroid~${M^*}$ with $ E(M^*)=E(M) $ whose bases are the complements of 
the bases of~$M$. For an  ${X \subseteq E}$, ${\boldsymbol{M  \upharpoonright X} :=(X,\mathcal{I} 
\cap \mathcal{P}(X))}$ is a matroid and it  is called the \emph{restriction} of~$M$ to~$X$.
We write ${\boldsymbol{M - X}}$ for $ M  \upharpoonright (E\setminus X) $  and call it the minor obtained by the 
\emph{deletion} of~$X$. 
The \emph{contraction} of $ X $ in $ M $ and the contraction of $ M $ onto $ X $ are 
${\boldsymbol{M/X}:=(M^* - X)^*}$ and $\boldsymbol{M.X}:= M/(E\setminus X) $ respectively.  
Contraction and deletion commute, i.e., for 
disjoint 
$ X,Y\subseteq E $, we have $ (M/X)-Y=(M-Y)/X $.  Matroids of this form are the  \emph{minors} of~$M$. The 
independence of an $ I\subseteq X $ in $ M.X $ is equivalent with $ I\subseteq \mathsf{span}_{M^{*}}(X\setminus I) $.  
If $ I $ is independent in $ M $  but $ I+e $ is dependent for some $ e\in E\setminus I $  then there is a unique 
circuit   $ \boldsymbol{C_M(e,I)} $ of $ M $ through $ e $ contained in $ I+e $.  We say~${X 
\subseteq E}$ \emph{spans}~${e \in E}$ in matroid~$M$ if either~${e \in X}$ or there exists a circuit~${C 
\ni e}$ with~${C-e \subseteq X}$. 
By letting $\boldsymbol{\mathsf{span}_{M}(X)}$ be the set of edges spanned by~$X$ in~$M$, we obtain a closure 
operation  
$ \mathsf{span}_{M}: \mathcal{P}(E)\rightarrow \mathcal{P}(E) $. 
An ${S \subseteq E}$ is \emph{spanning} in~$M$ if~${\mathsf{span}_{M}(S) = E}$. An $ S\subseteq X $ spans $ X $ in $ 
M.X $ iff $ X\setminus S $ is independent in $ M^{*} $. If $ M_i=(E_i, 
\mathcal{I}_i)$ is a matroid for $ i\in 
\Theta $ and the sets $ E_i $ are pairwise disjoint, then their direct sum is $ \boldsymbol{\bigoplus_{i\in 
\Theta}M_i}=(E,\mathcal{I}) $ where 
$ E=\bigsqcup_{i\in \Theta}E_i $ and $ \mathcal{I}=\{ \bigsqcup_{i\in \Theta}I_i : I_i\in \mathcal{I}_i\} $. For a class $ 
\mathfrak{C} $ of matroids $ \boldsymbol{\mathfrak{C}(E)} $ denotes the subclass $ \{ M\in \mathfrak{C}: E(M)=E \} $. 
A matroid is called uniform if for ever base $ B $ and every  edges $ e\in B $ and $ f\in E\setminus B $ the set $ B-e+f $ is also a 
base. Let 
$ \boldsymbol{U_{E,n}}$ be the $ n $-uniform matroid on $ E $, formally $ U_{E,n}:=(E , [E]^{\leq n})$. 

We need some further more subject-specific definitions. From now on let $ M 
$ and $ N $ be matroids on a common edge set $ E $.  We call a $ W\subseteq E $ an 
$ (M,N) $-\emph{wave} if  $ M\upharpoonright W $ admits an $ N.W $-independent base. Waves in the matroidal context were
introduced by 
Aharoni and Ziv in \cite{aharoni1998intersection} but it was also an important tool in the proof of the infinite version of Menger's 
theorem \cite{aharoni2009menger} by Aharoni and Berger. We write $ \boldsymbol{\mathsf{cond}(M,N)} $ for 
the condition: `For every $ (M,N) $-wave $ W $ there is an $ M $-independent base of $ N.W $.' A set $ I\in\mathcal{I}_M\cap 
\mathcal{I}_N $ is \emph{feasible} if $ 
\mathsf{cond}(M/I,N/I) $ holds.  
It is known (see Proposition \ref{wave union})  that there exists a $ \subseteq $-largest $ (M,N) $-wave which we denote  by 
$ \boldsymbol{W(M,N)} $.  Let $ \boldsymbol{\mathsf{cond}^{+}(M,N) }$ be the statement that $ W(M,N) $ consists of $ M 
$-loops and $ 
{r(N.W(M,N))=0} $. As the notation indicates it is a strengthening of $ \mathsf{cond}(M,N) $. Indeed, under  the assumption $ 
\mathsf{cond}^{+}(M,N) $,  $ \varnothing $ is an $ M $-independent base of $ N.W $ for every wave $ W $. A feasible $ I $ is 
called  \emph{nice} if $ \mathsf{cond}^{+}(M/I,N/I)  $ holds. For $ X\subseteq E $ let $ \boldsymbol{B(M,N,X)} $ be the 
(possibly empty) set of common bases of $ M \upharpoonright 
X $ and $ N.X $.
\section{Preliminary lemmas and preparation}\label{s: premil}
We collect  those necessary lemmas in this section that are either known from previous papers or follow more or less directly 
from definitions.

\subsection{Classical results}
The following two statements were proved by Edmonds' in \cite{edmonds2003submodular}:
\begin{prop}\label{prop: simult change}
        Assume that~$I$ is independent, 
        ${e_1, \dots, e_{m} \in \mathsf{span}(I) \setminus I}$ 
        and~${f_1, \dots, f_{m} \in I}$ 
        with ${f_j \in C(e_j, I)}$ 
        but ${f_j\notin C(e_k, I)}$ for~${k < j}$. 
        Then 
        \[
            {\left( I \cup \{e_1,\dots, e_{m}\} \right) \setminus \{f_1,\dots, f_{m}\}}
        \] 
        is independent and spans the same set as~$I$.
    \end{prop}    
    \begin{proof}
        We use induction on~$m$. 
        The case~${m = 0}$ is trivial.
        Suppose that~${m > 0}$. 
        On the one hand, the set ${I-f_m+e_m}$ is independent and spans the same set as~$I$. 
        On the other hand, ${C(e_j, I-f_m+e_m) = C(e_j, I)}$ for~${j < m}$ because ${f_m \notin C(e_j, I)}$ for~${j < m}$. 
        Hence by using the induction hypothesis for~${I-f_m+e_m}$ and ${e_1,\dots, e_{m-1}, f_1,\dots, f_{m-1}}$ we are done.
    \end{proof}
\begin{lem}[Edmonds' augmenting path method]\label{l: augP Edmonds}
For $ I\in \mathcal{I}_M\cap \mathcal{I}_N $, exactly one of the following statements holds:
\begin{enumerate}
\item There is a partition $ E=E_M\sqcup E_N $ such that $ I_M:=I\cap E_M $ spans $ E_M $ in $ M $ and  $ I_N:=I\cap E_N $ 
spans $ E_N $ in $ N $.  

\item There is a 
 $ P=\{ x_1,\dots, x_{2n+1} \}\subseteq E $  with $ x_{1}\notin \mathsf{span}_N(I) $ and $ 
 x_{2n+1}\notin 
 \mathsf{span}_M(I) $ such that $ I\vartriangle P\in 
 \mathcal{I}_M\cap \mathcal{I}_N $ with $ \mathsf{span}_{M}(I\vartriangle P) =\mathsf{span}_{M}(I+x_{2n+1}) $ and 
 $ \mathsf{span}_{N}(I\vartriangle P) =\mathsf{span}_{M}(I+x_{1}) $.
\end{enumerate}
\end{lem}
\noindent We will develop in Section 
\ref{s: augP}  a ``mixed'' augmenting path method which operates differently on the finitary 
and on the cofinitary part of an $ N\in (\mathfrak{F}\oplus \mathfrak{F}^{*})(E) $. The phrase `augmenting path' refers always 
to our mixed method except in the proof of Lemma \ref{one more edge}. 
Note that $ E_M $ is an $ (M,N) $-wave witnessed by $ I_M $ and $ E_N $ is an $ (N,M) $-wave witnessed by $ 
I_N $.

 One can define matroids in the language of circuits (see \cite{bruhn2013axioms}). The following claim is one of the axioms in 
 that case.    
\begin{claim}[Circuit elimination axiom]\label{Circuit elim}
Assume that $ C\ni e $ is a circuit and $ \{ C_x: x\in X \} $ is a family of circuits where $ X\subseteq C-e $ and $ C_x $ is a 
circuit with $C\cap X=\{ x \} $ avoiding $ e $. Then there is a circuit through $ e $ contained in 
 \[ \left( C\cup \bigcup_{x\in X}C_x \right) \setminus X =:Y \]
\end{claim}
\begin{proof}
Since $ C_x-x$ spans $ x $ we have $ C-e\subseteq\mathsf{span}(Y-e) $ and therefore $e\in \mathsf{span}(\mathsf{span}(Y-e) 
)$. But 
then $ e\in \mathsf{span}(Y-e) $ because $ \mathsf{span} $ is a closure operator.
\end{proof}
For finite matroids the axiom above is demanded only in the special case  where $ X $ is a singleton (known as ``Strong circuit 
elimination``) from which the case of arbitrary $ X $ can be derived by repeated application.
\begin{cor}\label{cor: Noutgoing arc}
Let $ I $ be an independent and suppose that there is a circuit $ C\subseteq \mathsf{span}(I) $ 
with $ e\in I\cap C $.  Then there is an $ f\in C\setminus I $ with $e\in  C(f,I) $.
\end{cor}
\begin{proof}
For every $ x\in C\setminus I $ we pick a circuit $ C_x $ with $ C_x\setminus I=\{ x \} $. If $ 
e\in C_x $ 
for some $ x $, 
then $ f:=x $ is as desired. Suppose for a contradiction that there is no such an $ x $. Then by Circuit elimination (Claim 
\ref{Circuit elim}) we obtain a circuit through $ e $ which is 
contained entirely in $ I $ contradicting the independence of $ I $.
\end{proof}

The following statement was shown by Aharoni and Ziv in \cite{aharoni1998intersection} using a slightly 
different terminology.
\begin{prop}\label{wave union}
The union of arbitrary many waves is a wave.
\end{prop}
\begin{proof}
Suppose that $ W_{\beta} $ is a wave  for $ \beta<\kappa $ and let 
$W_{<\alpha}:=\bigcup_{\beta<\alpha}W_{\beta}  $ for $ \alpha\leq \kappa $. We  fix a base 
$ B_{\beta} \subseteq W_{\beta} $
of $ M\upharpoonright W_{\beta} $ which is independent in $ N.W_{\beta} $.  Let us define 
$ B_{<\alpha} $  by transfinite recursion for $ \alpha\leq \kappa $ as follows.

\[B_{<\alpha}:= \begin{cases} \varnothing &\mbox{if } \alpha=0 \\
B_{<\beta}\cup (B_\beta \setminus W_{<\beta}) & \mbox{if } \alpha=\beta+1\\
\bigcup_{\beta<\alpha}B_{<\beta} & \mbox{if } \alpha \text{ is limit ordinal}. 
\end{cases} \]

 First we show  by transfinite induction
 that $ B_{<\alpha} $ is spanning in $ M\upharpoonright W_{<\alpha} $. For $ \alpha=0 $ it is trivial. 
 For a limit $ \alpha $ it follows directly from the induction hypothesis. If $ \alpha=\beta+1 $, then
 by the choice of $ B_\beta $, the set $ B_\beta \setminus W_{<\beta} $ spans $ W_{\beta}\setminus W_{<\beta} $ 
 in $ M/W_{<\beta} $. Since $ W_{<\beta} $ is spanned by $ B_{<\beta} $ in $ M $ by induction, it follows that 
 $ W_{<\beta+1} $ is spanned by $ B_{<\beta+1} $ in $ M $.
 
 The  independence of $ B_{<\alpha} $ in
 $ N.W_{<\alpha} $   can be reformulated as ``$W_{<\alpha}\setminus B_{<\alpha}$ is spanning in $ 
 N^{*} \upharpoonright 
 W_{<\alpha} $'', 
 which can be proved the same way as above. 
\end{proof}

\subsection{Some more recent results and basic facts}

\begin{thm}[Aigner-Horev, Carmesin and Frölich; Theorem 1.5 in \cite{aigner2018intersection}]\label{t: mixed}
If $ M\in \mathfrak{F}(E) $ and $ N\in 
\mathfrak{F}^{*}(E) $, then there is an $ I\in \mathcal{I}_M\cap \mathcal{I}_N $ and a partition $ E=E_M\sqcup E_N $ such 
that $ I_M:=I\cap E_M $ spans $ E_M $ in $ M $ and  $ I_N:=I\cap E_N $ spans $ E_N $ in $ N $.  
\end{thm}

\begin{cor}\label{cor: applyMixed}
If $ M\in \mathfrak{F}(E) $ and $ N\in 
\mathfrak{F}^{*}(E) $ satisfy $ \mathsf{cond}^{+}(M,N) $, then there is an $ M $-independent $ N $-base. 
\end{cor}
\begin{proof}
Let $ E_M, E_N, 
I_M $ and $ I_N $ be as in Theorem  \ref{t: mixed}. Then $ E_M $ is a wave witnessed by $ I_M $ thus by 
$ \mathsf{cond}^{+}(M,N) $ we know that $ E_M $ consists of $ M $-loops and $ r(N.E_M)=0 $. But then $ I_M=\varnothing $ 
and $ I_N $ is a base of $ N $ which is independent in $ M $.
\end{proof}

\begin{obs}\label{o: Mloop}
If $ \mathsf{cond}(M,N) $ holds and $ L $ is a set of the $ M $-loops, then $ r(N.L)=0 $ which means $ 
L\subseteq\mathsf{span}_N(E\setminus L) $.
\end{obs}
\begin{cor}\label{cor: Mloop}
If $ I $ is feasible, then $ r(N.(\mathsf{span}_{M}(I)\setminus I)=0$.
\end{cor}

\begin{obs}
If $ W_0 $ is an $ (M,N) $-wave and $ W_1 $ is an $ (M/W_0, N-W_0) $-wave, then $ W_0\cup W_1 $ is an $ (M,N) 
$-wave.
\end{obs}

\begin{cor}\label{cor: empty wave}
For $ W:=W(M,N) $, the largest $ (M/W, N-W) $-wave  is $ \varnothing $. In particular, $ \mathsf{cond}^{+}(M/W,N-W) $ 
holds.
\end{cor}

\begin{obs}\label{o: condRestrict}
 $ \mathsf{cond}^{+}(M,N) $ implies ${\mathsf{cond}^{+}(M\upharpoonright X,N.X) }$  for every $ X\subseteq E $.
\end{obs}
\begin{prop}\label{p: iterate feasible}
If $ I_0\in \mathcal{I}_N\cap \mathcal{I}_M $ and $ I_1 $ is feasible with respect to $ (M/I_0, 
N/I_0) $, then $ I_0\cup I_1 $ is 
feasible with respect to $ (M,N) $. If in addition $ I_1 $ is a nice feasible in regards to $ (M/I_0, N/I_0) $, then so is 
$ I_0\cup I_1 $ to $ (M,N) $.
\end{prop}
\begin{proof}
By definition the feasibility of $ I_1 $ w.r.t. $ (M/I_0, N/I_0) $ means  that the condition $ \mathsf{cond}(M/(I_0\cup 
I_1),N/(I_0\cup I_1)) $ 
holds. The feasibility of $ I_0\cup I_1 $ w.r.t. $ (M, N) $ means the same also by definition. For `nice feasible'  the argument is 
similar, only  $ \mathsf{cond} $ must be replaced by $ \mathsf{cond}^{+} $.
\end{proof}

The following lemma was introduced in  \cite{joo2020MIC}.
\begin{lem}\label{one more edge}
Condition $ \mathsf{cond}^{+}(M, N) $ implies that 
 whenever $ W $ is an $ (M/e, N/e) $-wave for some $ e\in E $ witnessed by $ B\subseteq W $, then 
 $B\in  B(M/e,N/e,W) $, i.e. $ B $ is spanning in $ N.W $.
\end{lem}
\begin{proof}
Let $ W $ be an $ (M/e, N/e)  $-wave. Note that 
$(N/e).W=N.W  $ by definition. Pick a $ B\subseteq W $ which is an $ N.W $-independent base of $ (M/e)\upharpoonright 
W $.  We may assume 
that 
$ e\in \mathsf{span}_M(W) $ and $ e $ is not an $ M $-loop. Indeed,  otherwise $  (M/e) \upharpoonright 
W=M\upharpoonright 
W$ holds 
and hence $ W $ is also an $ (M,N) $-wave. Thus by $ \mathsf{cond}^{+}(M, N) $ we may conclude that $ W $ consists of $ M 
$-loops  and hence $ B=\varnothing $, moreover, $ r(N.W)=0 $ and therefore $ 
\varnothing $ is a base of $ N.W $.

Then $ B $ is not a base of
$ M\upharpoonright W $  but ``almost'', namely $r(M/B \upharpoonright W) =1 $.  We apply the 
augmenting path Lemma \ref{l: augP Edmonds} by Edmonds  with $ B$ in regards to  $M\upharpoonright W$ and $  N.W  $. An 
augmenting 
path $ P $ cannot exist. Indeed, if $ P $ were an augmenting path then $ B\vartriangle P $ would show that $ W $ is an 
 $ (M,N) $-wave which does not consist of $ M $-loops, contradiction. Thus we get a partition 
$ W=W_{0}\sqcup W_{1} $ instead where $ W_0 $ is an $ (M \upharpoonright W,N.W) $-wave witnessed by $ B\cap W_0 $, 
and therefore also an $ (M,N) 
$-wave,  and $ W_1 $ is an 
$ (N.W, M\upharpoonright W) $-wave showed by $ B\cap W_1 $.
Then $ W_0 $ must consist of $ M $-loops by $ \mathsf{cond}^{+}(M, N) $ and 
therefore $ B\subseteq W_1 $ by the $ M $-independence of $ B $. We need to show that $ B $ is spanning not just in $ N.W_1 $ 
but also in $ N.W $. To do so let $ B' $ be a base of $ N-W$. 
Then $ B\cup B' $ spans $ E\setminus W_0 $ in $ N $ because $ B=B\cap W_1 $ is a base $ N.W_1 $. But 
then $ B\cup B' $ is spanning in  $ N $ because $ r(N.W_0)=0 $ by Observation \ref{o: Mloop}. We conclude that $ B $ is 
spanning in $ N.W $ as desired.
 \end{proof}

\subsection{Technical lemmas in regards to \texorpdfstring{$ \boldsymbol{B(M,N,W)} $}{BMNW}}

\begin{prop}\label{prop: WB nice feasible}
For $ W:=W(M,N) $, the elements of $ B(M,N,W) $ are nice feasible sets.
\end{prop}
\begin{proof}
Let $ B\in  B(M,N,W) $. Clearly $ B\in \mathcal{I}_M\cap \mathcal{I}_N $ because $ B\in \mathcal{I}_M\cap 
\mathcal{I}_{N.W} $ by definition and $ \mathcal{I}_{N.W} \subseteq \mathcal{I}_N $.  We know that $ W\setminus B $ is 
an $ 
(M/B, N/B) $-wave consisting of $ M/B $ loops and $ r(N.(W\setminus B))=0 $ because $ B $ is a 
base of $ N.W $. In order to show $ \mathsf{cond}^{+}(M/B, N/B) 
$, it is enough to prove that $ W\setminus B $ is actually  $ W(M/B, N/B)=:W' $. Suppose for a contradiction that $ W'\supsetneq 
W\setminus B$. Let $ B' $ be an $ N.W' $-independent base of $ M/B \upharpoonright W' $. Note that $ B\cup B' $ is a base of $ 
M \upharpoonright (W\cup 
W') $ and $ B'\cap (W\setminus B)=\varnothing $. Since $ B $ is $ N^{*} $-spanned by $ W\setminus B\subseteq (W\cup 
W')\setminus (B\cup B') $ and $ B' $ is $ N^{*} $-spanned by $ W'\setminus B'\subseteq (W\cup W')\setminus (B\cup B')  $, we 
may conclude that $ B\cup B' $ is independent in  $ N.(W\cup W') $. Thus $ W\cup W' $ is an $ (M,N) $-wave witnessed by $ 
B\cup B' $ which contradicts the maximality of $ W $. 
\end{proof}
\begin{cor}\label{cor: extNice}
Assume that $ I\in \mathcal{I}_M\cap \mathcal{I}_M $ and $ B\in B(M/I,N/I, W) $ where  $ W:=W(M/I,N/I) $.  Then $ I\cup B 
$ is a nice feasible set.
\end{cor}
\begin{proof}
Combine Propositions \ref{p: iterate feasible} and \ref{prop: WB nice feasible}.
\end{proof}

\begin{obs}\label{obs: cond+ loop delete}
If $ \mathsf{cond}^{+}(M,N) $ holds and $ L $ is a set of $ M $-loops, then $ \mathsf{cond}^{+}(M-L,N-L) $ also holds.
\end{obs}

 \begin{obs}\label{obs: common loops remove}
 Let $ W $ be a wave and let $ L $ be a set of common loops of $ M $ and $ N $. Then $ W\setminus L $ is also a wave and $ 
 B(M,N,W)=B(M,N,W\setminus L) $.
 \end{obs}
\begin{lem}\label{l: wave modify}
Let $ W $ be a wave and $  L\subseteq W $ such that $ L $ consists of $ M $-loops with $ r(N.L)=0 $. Then $ W\setminus L $ is 
an $ (M-L, N-L, W\setminus L) $-wave with \[ 
B(M,N,W)=B(M-L, N-L, W\setminus L). \]
\end{lem}
\begin{proof}
A set $ B $ is a base of  $ M\upharpoonright W $ iff it is a base of 
$ M\upharpoonright (W\setminus L) $ because $ L $ consists of $ M $-loops. A $ B\subseteq W\setminus L $ is $ N.W 
$-independent iff $ B\subseteq\mathsf{span}_{N^{*}}(W\setminus B) $.  This holds if and only if  $ 
B\subseteq\mathsf{span}_{N^{*}/L}(W\setminus 
(B\cup L)) $ i.e. $ B  $ is  independent in $ (N-L).(W\setminus L) $.   Note that $ r(N.L)=0 $ is equivalent with the $ 
N^{*} $-independence of $ L $. Thus for $ B\subseteq W\setminus L $,  $ W\setminus B $ is $ N^{*} $-independent iff 
$ W\setminus (B\cup L) $ is $ N^{*}/L $-independent. It means that $ B $ is spanning in $ N.W $ iff it is spanning in $ 
(N-L).(W\setminus L) $.

Thus the sets that are witnessing that $ W\setminus L $ is 
an $ (M-L, N-L, W\setminus L) $-wave are exactly those that are witnessing  that  $ W $ is an $ (M,N) $-wave, moreover, $ 
B(M,N,W)=B(M-L, N-L, W\setminus L) $ holds.

\end{proof}

\begin{lem}\label{l: minorsChanged}
Assume that $ X_j, Y_j \subseteq E $ for $ j\in \{ 0,1 \} $ where $ X_j\sqcup Y_j=Z $ for $ j\in \{ 0,1 \} $, furthermore 
$ \mathsf{span}_M(X_0)=\mathsf{span}_M(X_1) $ and $ \mathsf{span}_{N^{*}}(Y_0) =\mathsf{span}_{N^{*}}(Y_1) $. 
Then for every $ X\subseteq  E\setminus Z $ we have\[ B(M/X_0-Y_0, N/X_0-Y_0, X)=B(M/X_1-Y_1, N/X_1-Y_1, X). \]
\end{lem}
\begin{proof}
The matroids $ M/X_0-Y_0 $ and  $ M/X_1-Y_1$ are the same as well as the matroids $ N/X_0-Y_0$ and  $ N/X_1-Y_1 $.
\end{proof}

\section{Reductions}\label{s: reduction}
We repeat here our main results for convenience:
\PC*

\MI*

\subsection{Matroid Intersection with a finitary \texorpdfstring{$\boldsymbol{M} $}{M}}
As we mentioned, the  method by Bowler and Carmesin used to prove  Corollary 3.9 in \cite{bowler2015matroid} works not only 
for the class of all matroids but can be adapted for every class  closed under certain operations. We apply their 
technique to obtain the following reduction:

\begin{lem}\label{l: M finitary}
Theorems \ref{t: main result0} and \ref{t: main result} are implied by the special case of Theorem \ref{t: main result} where $ 
M\in \mathfrak{F} $.
\end{lem}

\begin{proof}
First we show that one can assume without loss of generality in the proof of Theorem \ref{t: main result0} that $ \Theta  $ is 
countable. To 
do so let  \[ E':=\{ e\in E: \left|\{i\in 
\Theta: \{ e \}\in \mathcal{I}_{M_i} \}\right|\leq \aleph_0 \} \]  and 
\[  \Theta':= \{ i\in \Theta: (\exists e\in E' ) (\{ e \}\in \mathcal{I}_{M_i}) \}. \]
We apply Theorem \ref{t: main result0} with $ E' $ and with the  countable family  $(M_i\upharpoonright E': i\in \Theta') $. 
Then we obtain a partition 
$ E'=E'_p\sqcup E'_c $ such that $ (M_i \upharpoonright E'_p: i\in \Theta') $ admits a packing $ (S_i: i\in \Theta') $ and 
$ (M_i\upharpoonright E'. E_c: i\in \Theta') $ admits a covering $ (I_i: i\in \Theta') $. Let $ E_p:=E'_p $ and $ E_c:= E\setminus 
E_p=E'_c\cup (E\setminus E')  $. By construction $ r_{M_i}(E')=0 $ for $ i\in 
\Theta \setminus \Theta' $. Thus by letting $ S_i:=\varnothing $  for $ i\in \Theta \setminus \Theta' $ the 
family $ (S_i: i\in \Theta) $ is a packing w.r.t. $ (M_i \upharpoonright E_p: i\in \Theta) $.  
Let  $ g: E\setminus E'\rightarrow 
\Theta\setminus \Theta' $ be injective. For $ i\in \Theta \setminus \Theta'  $, we take $ I_i:=\{ g^{-1}(i) \}$ if $ i\in 
\mathsf{ran}(g) $ and $ I_i:=\varnothing $ otherwise. Then $ (I_i: i\in \Theta) $ is a covering for $ (M_i. E_c: i\in \Theta) $ and 
we are done.

We proceed with the proof of Theorem \ref{t: main result0} assuming  that $ \Theta $ is countable.  For $ i\in \Theta $, let $ M'_i 
$ be the matroid on $ 
E\times \{ i \} $ that we 
obtain by 
``copying'' $ M_i $ via the bijection $ e\mapsto (e,i) $.  Then for

\[M:= \bigoplus_{i\in \Theta}M_i',\ \text{ and } N:= \bigoplus_{e\in E}U_{ \{ e \}\times \Theta, 1  } \]
we have $M\in ( \mathfrak{F}\oplus \mathfrak{F}^{*})(E\times \Theta) $ and $ N\in \mathfrak{F}(E\times \Theta) $ where $ 
E\times \Theta $ is countable. Thus by assumption there is a partition $ 
E\times \Theta=E_M\sqcup E_N $ and an $ I\in \mathcal{I}_{M}\cap \mathcal{I}_{N} $ such that $ I_M:=I\cap E_M $ 
spans $ E_M $ in $ M$ and $ I_N:=I\cap E_N $ 
spans $ E_N$ in $ N$. The  $ M $-independence of $ I $ ensures that $ J_i:=\{ e\in E: (e,i)\in I \} $ is $ M_i $-independent. 
The $ N $-independence of $ I $ guarantees  that the sets $ J_i $ are pairwise disjoint. Let
$ E_c:=\{ e\in E: (\exists i\in \Theta) (e,i)\in E_N \} $. Then for each $ e\in E_c $ there must be some $ i\in \Theta $ with $ 
(e,i)\in 
I_N $ because $ E_N\subseteq \mathsf{span}_N(I_N) $. Thus the sets $ J_i $ cover $ E_c $ and so do the sets $ 
I_i:=J_i\cap E_c $. It is enough to show that $ S_i:=J_i\setminus I_i $ spans $ E_p:=E\setminus E_c $ in $ M_i $ for every $ 
i\in \Theta $. Let $ f\in E_p $ and $ i\in \Theta $ be given.  Then $ \{ f \}\times \Theta \subseteq E_M $ follows directly from the 
 definition of $ E_p $, in particular  $ (f,i)\in E_M $. We know that $(f,i)\in  
  \mathsf{span}_{M}(I_M)$ and hence  $ f\in \mathsf{span}_{M_i}(\{ e\in E: (e,i)\in I_M \}) $. Suppose for a contradiction that 
  $h\in E_c\cap \{ e\in E: (e,i)\in I_M \} $. Then for some $ j\in 
  \Theta $ we have $ (h,j)\in E_N $. Since $ (h,j)\in \mathsf{span}_N(I_N) $, we have $ (h,k)\in I_N $ for some $ k\in \Theta $. 
  But then $ (h,i), (h,k)\in I $ are distinct elements thus $ i\neq k $ which contradict the $ N $-independence of $ I $. Therefore $ 
  E_c\cap \{ e\in E: (e,i)\in I_M \}=\varnothing $. Since $ \{ e\in E: (e,i)\in I_M  \}\subseteq J_i $  by the definition of $ J_i 
  $ we conclude $  \{ e\in E: (e,i)\in I_M \}=S_i $. Therefore $ (S_i: i\in \Theta) $ is a packing for $ (M_i \upharpoonright E_p: 
  i\in \Theta) $ and $ (I_i: i\in \Theta) $  is a covering for $ (M_i.  E_c: i\in \Theta) $ as desired. 

Now we derive Theorem \ref{t: main result} from Theorem \ref{t: main result0}. To do so, we take a partition $E=E_p\sqcup 
E_c$ such that $ (S_M, S_N) $ is a packing for 
 $ (M\upharpoonright E_p, N^{*}\upharpoonright E_p) $ and $ (R_M, R_N) $ is a covering for $ (M.E_c, N^{*}.E_c ) $.  Let $ 
 I_M\subseteq S_M $ be a base of  $ M\upharpoonright E_p $ and we define $ I_N:= R_M $. By construction 
  $E_p\subseteq \mathsf{span}_M(I_M)  $ and $ I_N\in \mathcal{I}_{M.E_c} $.    We also know that
 \[ I_M\subseteq \mathsf{span}_{N^{*}}(S_N) \subseteq \mathsf{span}_{N^{*}}(E_p\setminus I_M) \]
 which means $ I_M\in \mathcal{I}_{N.E_{p}} $. 
 Finally, $R_N\in \mathcal{I}_{N^{*}.E_c} $ means that $ E_c\setminus R_N $ spans $ E_c $ in $ N $ and therefore so does $ 
 I_N=R_M\supseteq E_c\setminus R_N $.
 \end{proof}
\subsection{Finding an \texorpdfstring{$ \boldsymbol{M} $}{M}-independent base of  \texorpdfstring{$ \boldsymbol{N} 
$}{N}}
The following reformulation of the matroid intersection problem was introduced by Aharoni and Ziv in 
\cite{aharoni1998intersection} but its analogue by Aharoni was already an important tool to attack (and eventually solve in 
\cite{aharoni2009menger}) the Erdős-Menger Conjecture.
\begin{restatable}{claim}{indepB}\label{c: M-indep N-base}
  Assume that  $ M\in \mathfrak{F}(E) $ and $ N\in (\mathfrak{F}\oplus\mathfrak{F}^{*})(E) $ such that $ E $ is countable and 
   $ \mathsf{cond}^{+}(M,N) $ holds. Then there is an $ M $-independent base of $ N $.
\end{restatable}
\begin{lem}\label{l: reduc2}
Claim \ref{c: M-indep N-base} implies our main results Theorems \ref{t: main result0} and \ref{t: main result}. 
\end{lem}
\begin{proof}
By Lemma \ref{l: M finitary} it is enough to show that the special case of Theorem \ref{t: main result} where $ M\in 
\mathfrak{F} $ follows from Claim \ref{c: M-indep N-base}.
To do so, let  $ E_M:=W(M,N) $ and let $ I_M\subseteq E_M $ be a witness that $ E_M $ is a wave. For $ E_N:=E\setminus 
E_M $, we have $ 
M/E_M\in \mathfrak{F}(E_N) $  and 
$ N-E_M\in (\mathfrak{F}\oplus \mathfrak{F}^{*})(E_N) $, furthermore, $ \mathsf{cond}^{+}(M/W,N-W) $ holds (see 
Corollary \ref{cor: empty wave}). By Claim \ref{c: M-indep N-base}, there is an $ M/E_M $-independent base $ I_N $ of $ 
N-E_M $. Then $ I\in \mathcal{I}_M\cap \mathcal{I}_{N}$ and $ E=E_M\sqcup E_N $ such that $ I_M:=I\cap E_M $ spans 
  $ E_M $ in $ M $ and  $ I_N:=I\cap E_N $ spans $ E_N $ in $ N $ as desired. 
\end{proof}
\subsection{Reduction to a key-lemma}
From now on we assume that $ M\in \mathfrak{F}(E) $ and $ N\in (\mathfrak{F}\oplus\mathfrak{F}^{*})(E) $ where $ E $ is 
countable.  Let $ 
\boldsymbol{E_0} $ be 
the union of the finitary components of $ N $ and let  
 $ \boldsymbol{E_1}:=E\setminus E_0 $.  Note that $ N\upharpoonright E_0 $ is finitary,  
$ N\upharpoonright E_1 $ is cofinitary and no $ N $-circuit meets both $ E_0 $ and $ E_1 $. 

\begin{restatable}[key-lemma]{lem}{keylemma}\label{l: key-lemma}
  If $ \mathsf{cond}^{+}(M,N) $ holds, then for every $ e\in E_0 $  there is a nice 
  feasible $ I $ with $ e\in \mathsf{span}_{N}(I) $.
\end{restatable}

\begin{lem}
Lemma \ref{l: key-lemma}  implies  our main results Theorems \ref{t: main result0} and \ref{t: main result}. 
\end{lem} 
\begin{proof}
It is enough to show that Lemma \ref{l: key-lemma} implies Claim \ref{c: M-indep N-base} because of  Lemma \ref{l: reduc2}.  
Let us 
fix an enumeration $ \{ e_n: n\in 
\mathbb{N} \} $ of $ E_0 $. 
We build an $ \subseteq $-increasing sequence $ (I_n)  $  of  nice feasible sets starting 
with $ I_0:=\varnothing $ (who is nice feasible by $ \mathsf{cond}^{+}(M,N) $) in such a way that $ {e_n\in  
\mathsf{span}_N(I_{n+1})} $. Suppose that  $ I_n $ is already defined. 
If $ e_n\notin \mathsf{span}_N(I_n) $, then we 
apply Lemma \ref{l: key-lemma} with $ (M/I_n, N/I_n) $ and $ e_n $ and take the union of the resulting $ I $ with  $ I_{n} $ to 
obtain $ I_{n+1} $ (see Observation \ref{p: iterate feasible}), otherwise let $ I_{n+1}:=I_n $. The recursion is done. Now we 
construct an $ M $-independent  
$ I^{+}_n \supseteq I_n $  with $ E_1\subseteq  \mathsf{span}_N(I^{+}_n) $ for $ n\in \mathbb{N} $. The matroid $ M/I_n 
\upharpoonright 
(E_1\setminus I_n) $ is finitary and $ N.(E_1\setminus I_n)=(N\upharpoonright E_1)/(I_n\cap E_1)$ is cofinitary, moreover, by 
Observation \ref{o: condRestrict} 
\[ \mathsf{cond}^{+}(M/I_n, N/I_n) \Longrightarrow \mathsf{cond}^{+}(M/I_n \upharpoonright (E_1\setminus I_n),  
N.(E_1\setminus I_n )).  \] Thus by Corollary \ref{cor: applyMixed} there is an $ M/I_n $-independent base $ B_n $ of 
$ (N\upharpoonright E_1)/(E_1\cap I_n) $ and $ 
I^{+}_n:=I_n\cup B_n $ is as desired. 

Let $ \mathcal{U} $ be a free ultrafilter on $ \mathcal{P}(\mathbb{N}) $ ad we define 
\[ S:=\{ e\in E: \{ n\in \mathbb{N}: e\in I^{+}_n \}\in \mathcal{U} \}. \] Then $ S $ is $ M 
$-independent and $ N $-spanning and therefore we are done. Indeed, suppose for a contradiction that $ S $ contains an $ M 
$-circuit $ C $. For $ e\in C $, we pick a $ U_e\in 
\mathcal{U} $ with $ e\in I^{+}_n $ for $ n\in U_e $. Since $ M $ is finitary, $ C $ is finite, thus $U:= \cap \{ U_e: e\in C \}\in 
\mathcal{U} $. 
But then for $ n\in U $ we have $ C\subseteq I^{+}_n $ which contradicts the $ M $-independence of $ I^{+}_n $. Clearly, $ 
I_{n+1}\subseteq S $ for every $ n\in \mathbb{N} $ and therefore $ E_0\subseteq \mathsf{span}_N(S) $. Finally, suppose for a 
contradiction that there is some 
$ N^{*}\upharpoonright E_1 $-circuit $ C' $ with $ S\cap C'=\varnothing $. For $ e\in C' $, we can pick a $ U_e'\in 
\mathcal{U} $ with $ e\notin I^{+}_n $ for $ n\in U_e' $.   Since $ N^{*}\upharpoonright E_1 $ is finitary, $ C' $ is finite, thus 
$U':= \cap \{ U_e': e\in 
C' \}\in 
\mathcal{U} $. 
But then for $ n\in U' $ we have $  I^{+}_n\cap C=\varnothing $ which contradicts   $E_1\subseteq 
\mathsf{span}_N(I^{+}_n) $.
\end{proof}
\section{Mixed augmenting paths}\label{s: augP}
In the section we  introduce an `augmenting path' type of method and analyse it in order to show some properties we need later. 
On $ E_0 $ the definition will be
the same  as in the  proof of 
the Matroid Intersection Theorem by Edmonds \cite{edmonds2003submodular} but on $ E_1 $ we need to define it in a different 
way considering that $ N\upharpoonright E_1 $ is cofinitary. For brevity we write 
$ \boldsymbol{\overset{\circ}{\mathsf{span}}_{M}(F)} $ for $ \mathsf{span}_{M}(F)\setminus F $  
and  $ \boldsymbol{F^{j}} $ for $F\cap E_j $ 
where $ F\subseteq E $ and $ j\in \{ 0,1 \} $. 
We call an $F\subseteq E $ \emph{dually safe} if  $ F^{1} $ is spanned by $ 
\overset{\circ}{\mathsf{span}}_M(F) $ in $ N^{*} $. 

\begin{lem}\label{l: enoughAddB}
If $ I\in \mathcal{I}_M\cap \mathcal{I}_n $ is dually safe and  $ B\in B(M/I, N/I, W) $ for $ W:=W(M/I, N/I) $, then 
$ I\cup B $ is a nice dually safe feasible set.
\end{lem}
\begin{proof}
We already know by Corollary \ref{cor: extNice} that $ I\cup B $ is a nice feasible set. By using that $ I $ is dually safe and $ 
\overset{\circ}{\mathsf{span}}_M(I)\subseteq\overset{\circ}{\mathsf{span}}_M(I\cup B) $ we get
\[ I^{1}\subseteq \mathsf{span}_{N^{*}}(\overset{\circ}{\mathsf{span}}_M(I))\subseteq  
\mathsf{span}_{N^{*}}(\overset{\circ}{\mathsf{span}}_M(I\cup B)). \]

Since $ B\in B(M/I, N/I, W) $ we have $ W\setminus B\subseteq \overset{\circ}{\mathsf{span}}_M(I\cup B) $ and $ B $ is a 
base of $ N.W $. The latter can be rephrased as `$ W\setminus B $ is a base of $ N^{*}\upharpoonright W $'. By combining 
these $ \overset{\circ}{\mathsf{span}}_M(I\cup B) $  spans $ B $ in $ N^{*} $. Therefore 
$(I\cup B)\cap E_1 \subseteq \mathsf{span}_{N^{*}}(\overset{\circ}{\mathsf{span}}_M(I\cup B)) $, which means that $ I\cup 
B $ is dually safe. 
\end{proof}
\begin{prop}\label{prop: IE_1 common base}
For a  dually safe feasible $ I $, $ \overset{\circ}{\mathsf{span}}_{M}(I)^{1}  $ is 
a base of $ {N^{*}\upharpoonright \mathsf{span}_{M}(I)^{1} } $. 
\end{prop}
\begin{proof}
By the definition   of `dually safe', $ \overset{\circ}{\mathsf{span}}_{M}(I)^{1}  $ spans $ N^{*}\upharpoonright 
\mathsf{span}_{M}(I)^{1}  $. Furthermore, $ r(N. \overset{\circ}{\mathsf{span}}_{M}(I))=0 $ by Corollary \ref{cor: Mloop},  
which is equivalent with the $ N^{*} $-independence of $ 
\overset{\circ}{\mathsf{span}}_{M}(I) $.
\end{proof}

For a  dually safe  feasible $ I $, 
we define an auxiliary digraph $ D(I) $ on $ E $.
Let $ xy $ be an arc of $ D(I) $ iff
one of the following possibilities occurs:
\begin{enumerate}
\item\label{item: D(I) 1} $ x\in E\setminus I $ and $ I+x $ is $ M $-dependent with $ y\in C_M(x, I)-x $,
\item\label{item: D(I) 2} $x\in I^{0}  $ and $ C_{N}(y,I) $ is well-defined and contains $ x $,
\item\label{item: D(I) 3} $x\in I^{1} $ and 
$ y\in C_{N^{*}}(x,\overset{\circ}{\mathsf{span}}_{M}(I)^{1} ) -x $ (see 
Proposition \ref{prop: IE_1 common base}).
\end{enumerate}

An augmenting path for  a nice dually safe feasible $ I $ is a $ P=\{ x_1,\dots, x_{2n+1} \} $ where 
\begin{enumerate}
[label=(\roman*)]
\item $ x_1\in  E_0\setminus \mathsf{span}_N(I) $,
\item $ x_{2n+1}\in E_0\setminus \mathsf{span}_M(I)  $,
\item $ x_kx_{k+1}\in D(I) $ for $ 1\leq k\leq  2n $,
\item\label{item: no jumping} $  x_kx_{\ell}\notin D(I) $ if $ k+1<\ell  $.
\end{enumerate}

\begin{prop}\label{p: augpath}
If $ I $ is a  dually safe feasible set and $ P=\{ x_1,\dots, x_{2n+1} \} $ is an augmenting path for $ I $, then 
$ I\vartriangle P$ is a dually safe element of $ \mathcal{I}_M\cap \mathcal{I}_N $  with
\begin{enumerate}
[label=(\Alph*)]
\item\label{item: A} $\mathsf{span}_{M}(I\vartriangle P)=\mathsf{span}_{M}(I+x_{2n+1}) $,
\item\label{item: B} $\mathsf{span}_{N}(I\vartriangle P)\cap E_0=\mathsf{span}_{N}(I+x_1)\cap E_0 $,
\item\label{item: C} $\mathsf{span}_{N^{*}} (\overset{\circ}{\mathsf{span}}_{M}(I)^{1})=\mathsf{span}_{N^{*}} 
(\overset{\circ}{\mathsf{span}}_{M}(I)^{1}\vartriangle P^{1}) $ and $ 
\overset{\circ}{\mathsf{span}}_{M}(I)^{1}\vartriangle P^{1}\in \mathcal{I}_{N^{*}} $.
\end{enumerate}
\end{prop}
\begin{proof}

The set $ I+x_{2n+1} $ is $ M $-independent by the definition of $ P $. Property \ref{item: no jumping} ensures that we can 
apply Proposition \ref{prop: simult change} with $I+x_{2n+1},\  e_j=x_{2j-1},\ f_j=x_{2j}\ (1\leq j\leq n) $ and $ M $ and 
conclude that 
$ I\vartriangle P\in \mathcal{I}_M $ and \ref{item: A} holds. 

To prove $ (I\vartriangle P)\cap E_0\in \mathcal{I}_N $ and 
\ref{item: B} we 
proceed similarly. We start with the $ N $-independent set $ (I+x_{1})\cap E_0 $.  In order to satisfy the premisses of 
Proposition \ref{prop: simult change} via property 
\ref{item: no jumping}, we need to enumerate the relevant edge 
pairs backwards. 
Namely, for $j\leq \left| I^{0}\cap P \right|$ let $ e_j:=x_{i_j+1} $ where  $ i_j $ is the $ j $th largest index with $ x_{i_j}\in 
I^{0} $ and $ f_j:=x_{i_j} $. We conclude that $ (I\vartriangle P)\cap E_0\in \mathcal{I}_N 
$ and \ref{item: B} 
holds. 

Finally, we let  $ e_j:=x_{i_j} $ for $ j\leq \left|I^{1}\cap P\right| $ where  $ i_j $ is the $j$th smallest  index with $ 
x_{i_j}\in I^{1} $ and $ 
f_j:=x_{i_j+1} $. 
Recall  that  $ \overset{\circ}{\mathsf{span}}_{M}(I)^{1} $ is $ N^{*} $-independent 
(see Proposition  \ref{prop: IE_1 common base}). 
We  apply Proposition \ref{prop: simult change} with 
$\overset{\circ}{\mathsf{span}}_{M}(I)^{1} ,\  e_j,\ f_j $ and $ N^{*} $ to conclude \ref{item: C}. This means that  $ 
\overset{\circ}{\mathsf{span}}_{M}(I)^{1}\vartriangle P^{1} $ is a base of $ N^{*}\upharpoonright  
\mathsf{span}_{M}(I)^{1}$ because $ \overset{\circ}{\mathsf{span}}_{M}(I)^{1} $ was a base 
of it by Proposition \ref{prop: IE_1 common base}.  By \ref{item: A} and by the definition of $ P $ we 
know that \[ (I\vartriangle P)\cap 
E_1\subseteq 
I^{1}\cup P^{1}\subseteq \mathsf{span}^{1}_{M}(I). 
\]  By combining these we obtain \[ (I\vartriangle P)\cap 
E_1\subseteq \mathsf{span}_{N^{*}}(\overset{\circ}{\mathsf{span}}_{M}(I)^{1}\vartriangle P^{1}). \] The set  $ 
I\vartriangle 
P  $ is disjoint from  
$ \overset{\circ}{\mathsf{span}}_{M}(I)\vartriangle P$ because $ I $ is disjoint 
from $ 
\overset{\circ}{\mathsf{span}}_{M}(I) $, moreover, $ \overset{\circ}{\mathsf{span}}_{M}(I)\vartriangle P $ contained in $ 
\mathsf{span}_{M}(I\vartriangle P) $. Hence  $ \overset{\circ}{\mathsf{span}}_{M}(I\vartriangle P)  $ 
contains $ 
\overset{\circ}{\mathsf{span}}_{M}(I)^{1}\vartriangle P^{1} $ and therefore $ N^{*} $-spans 
$ (I\vartriangle P) \cap E_1 $, i.e. $ I\vartriangle P $ is dually safe. It means that $ (I\vartriangle P)\cap E_1 $ is independent in $ 
N.\mathsf{span}_{M}(I\vartriangle P)^{1} $. 
Thus uniting   $ (I\vartriangle P)\cap E_0\in \mathcal{I}_N $ 
with $ (I\vartriangle P)\cap E_1 $ preserves $ N 
$-independence, i.e. $ I\vartriangle P\in \mathcal{I}_N $. 

\end{proof}
\begin{lem}\label{l: aug extend}
If $ I $ is a nice dually safe feasible set and $ P=\{ x_1,\dots, x_{2n+1} \} $ is an augmenting path for $ I $, then 
$ I \vartriangle P $ can be extended to a nice dually safe feasible set.
\end{lem}
\begin{proof}
Let $ M':= M/(I\vartriangle P)$ and $ N':=N/(I\vartriangle P) $.  By Lemma \ref{l: enoughAddB} it is enough to show that 
$ B(M', N', W)\neq \varnothing $ for $ 
W:=W(M', N') $.  Statements \ref{item: A} and \ref{item: B} of Proposition \ref{p: augpath}  ensure that the elements of  $ L:= 
\{ 
x_1, x_3,\dots, 
x_{2n-1} \}\cap E_0 $ are common 
loops of  $ M'$ and $ N' $.  Statement \ref{item: C} tells that $ Y_0:=\overset{\circ}{\mathsf{span}}_{M}(I)^{1} 
$ and $ 
Y_1:=\overset{\circ}{\mathsf{span}}_{M}(I)^{1}\vartriangle P^{1} $  have the same $ N^{*} 
$-span, furthermore, $ Y_1 $ is $ N^{*} $-independent, i.e. $ r(N.Y_1)=0 $. Note  that $ Y_1 $ consists of $ M' $-loops by 
\ref{item: A}.
Thus by applying Observation \ref{obs: common loops remove} with $ W $ and $ L $ and then Lemma \ref{l: wave modify} 
with  $ W\setminus L $  and $ Y_1 $  we can conclude  that $ W\setminus (L\cup Y_1) $ is an 
$ (M'-Y_1, N'-Y_1) $-wave, furthermore, 
\[ B(M',N', W)=B(M'-Y_1, N'-Y_1, W\setminus (L\cup Y_1)). \]
The sets $X_0:= I+x_{2n+1} $ and $X_1:= I\vartriangle P $ have the same $ M $-span (see  
Proposition \ref{p: augpath}/ \ref{item: A}). Recall 
that $ M'=M/X_1 $ and 
$ 
N'=N/X_1 $ by definition. Hence Lemma \ref{l: minorsChanged} 
ensures that $ W\setminus (Y_1\cup L) $ is also an $ (M/X_0-Y_0, M/X_0-Y_0) $-wave with 
\[ B(M/X_1-Y_1, N/X_1-Y_1, W\setminus (Y_1\cup L))=B(M/X_0-Y_0, M/X_0-Y_0, W\setminus (Y_1\cup L)).  \]
We have $ \mathsf{cond}^{+}(M/I, N/I) $ because $ I $ is a nice feasible set by assumption. Then by 
Observation \ref{obs: cond+ loop delete}
$ \mathsf{cond}^{+}(M/I-Y_0, N/I-Y_0) $ also holds. Applying Lemma \ref{one more edge} with $ M/I-Y_0,\  N/I-Y_0,\ 
W\setminus(Y_0\cup L)   $ and $ x_{2n+1}$ tells
$ B(M/X_0-Y_0, M/X_0-Y_0, W\setminus (Y_1\cup L))\neq \varnothing $ which completes the proof.
\end{proof}

\begin{lem}\label{l: arc remain lemma}
If  $ P=\{ x_1, \dots, x_{2n+1} \} $ is an augmenting path for $ I $ which contains neither  $x$ nor any of its out-neighbours in $ 
D(I) $, then $xy\in  D(I) $ implies $ xy \in D(I \vartriangle P) $.
\end{lem}
\begin{proof}
Suppose that $ xy\in D(I) $. First we assume that $ x\notin I $. Then the set of the out-neighbours of $ x $ is  
 $ C_M(x,I)-x $. By 
assumption $ P\cap C_M(x,I)=\varnothing $ and 
therefore $C_M(x,I)\subseteq  I \vartriangle P$ thus $ C_M(x,I)=C_M(x,I \vartriangle P) $. This means by definition that $ x $ 
has the same out-neighbours in 
$ D(I) $ and $ D(I \vartriangle P) $.

 If $ x\in I^{1} $, then we can argue similarly.  The set of the out-neighbours of $ x $ in $ D(I) $ is 
$ C_{N^{*}}(x,\overset{\circ}{\mathsf{span}}_{M}(I)^{1} ) -x $. By 
assumption $ P\cap C_{N^{*}}(x,\overset{\circ}{\mathsf{span}}_{M}(I)^{1} )=\varnothing $ and 
therefore  $ C_{N^{*}}(x,\overset{\circ}{\mathsf{span}}_{M}(I)^{1} ) \cap  (I \vartriangle P)=\{ x\} $.  Since $ 
\mathsf{span}^{1}_{M}(I\vartriangle P) \supseteq 
\mathsf{span}^{1}_{M}(I)$ because of Proposition \ref{p: augpath}/\ref{item: A}, we also have $ 
C_{N^{*}}(x,\overset{\circ}{\mathsf{span}}_{M}(I)^{1} )\subseteq 
\mathsf{span}^{1}_{M}(I\vartriangle P) $. By combining these we conclude
 
 \[ C_{N^{*}}(x,\overset{\circ}{\mathsf{span}}_{M}(I\vartriangle P)^{1} ) 
 =C_{N^{*}}(x,\overset{\circ}{\mathsf{span}}_{M}(I)^{1} ). \] This means 
by definition that $ x $ has the same out-neighbours in $ D(I) $ and $ D(I \vartriangle P) $.

We turn to the case where $ x\in I^{0} $. By definition $ C_N(y,I) $ is well-defined and
contains $ x $,  in particular $ y\in E_0 $.
For $ k\leq n $, let us denote 
$ I+x_1-x_2+x_3-\hdots -x_{2k}+x_{2k+1} $ by $ I_k $. Note  that
$ {I_n=I \vartriangle P} $. We show by induction on $ k $ that $ I_k $ is $ N $-independent and
$ {x\in C_N(y, I_k)} $. Since
$ I+x_1 $ is $ N $-independent by definition and $ x_1\neq y $ because $ y\notin P$ by assumption, we obtain $ 
C_N(y,I)=C_N(y,I_0) $, thus for $ k=0 $ it holds.  Suppose that $ n>0 $ and
we already know the statement for some  $ k<n $. We have $ C_N(x_{2k+3},I_k)=C_N(x_{2k+3},I)\ni x_{2k+2} $ because 
there is no ``jumping arc''   in the augmenting path by property \ref{item: no jumping}.  It  follows via the $ N $-independence of 
$ I_k $ that  $ I_{k+1} $ is also $ N 
$-independent.
If $ x_{2k+2}\notin C_N(y, I_k) $ then $ C_N(y, I_k)=C_N(y, I_{k+1}) $ 
and the induction step is done. Suppose that   $ x_{2k+2}\in C_N(y, I_k) $. Then $x_{2k+2}, x_{2k+3}\in E_0  $, moreover, 
$x\notin C_N(x_{2k+3},I) $ since 
otherwise $P$ would contain the out-neighbour $x_{2k+3} $ of $ x $ in $ D(I) $. We apply  circuit elimination (Claim 
\ref{Circuit elim})
with  $C= C_N(y, I_k),\ e=x,\  X=\{ x_{2k+2} \},\  C_{x_{2k+2}}=C_N(x_{2k+3},I_k) $. The resulting circuit $ C'\ni x $ can 
have at most one element out of $ I_{k+1} $, namely $ y $. Since $ I_{k+1} $ is $ N $-independent, there must be  at least one 
such an element and therefore $ C'=C_N(y,I_{k+1}) $.
\end{proof}
\begin{obs}\label{arc remain fact}
If $ xy \in D(I) $ and  $ J\supseteq I $ is a dually safe feasible set with
$ \{ x,y \}\cap J=\{ x, y \}\cap I $, then $ xy \in D(J) $ (the same circuit is the witness).
\end{obs}

\section{Proof of the key-lemma}\label{s: proof of key-lemma}

\keylemma*
\begin{proof}
It is enough to build a  sequence $(I_n)$ of nice dually safe feasible sets  such that 
$ (\mathsf{span}_N(I_n)\cap E_0) $ is  an ascending sequence  exhausting $ E_0 $.   We fix a well-order $ \boldsymbol{\prec} $ 
of type $ \left|E_0\right| $ on  $ E_0 $. Let $ I_0=\varnothing $, which is a nice dually safe feasible set by  $ 
\mathsf{cond}^{+}(M,N) $. 
Suppose that  $ I_n $ is already defined. If there is no augmenting path for $ I_n $, then we let $ I_m:=I_n $ for $ m>n $. 
Otherwise we pick an augmenting path $ P_n $ for $ I_n $ in such a way that its first element is as $ \prec $-small as possible. 
Then we  apply Lemma \ref{p: augpath} to extend $ I\vartriangle P $  to a nice dually safe feasible set which we define to be $ 
I_{n+1} $. The recursion is done.

   Let $ \boldsymbol{X}:=E\setminus \bigcup_{n\in \mathbb{N}}\mathsf{span}_{N}(I_n) $ and for $ x\in X $, let $ 
   \boldsymbol{E(x,n)} $ be the set of elements that are 
reachable from $ x $ in $ D(I_n) $ by a directed path. We define $\boldsymbol{n_x} $ to be the smallest 
natural number such that  for every $ y\in E\setminus X $ with   $ y \prec x $ we have $ y\in \mathsf{span}_N(I_{n_x}) $.
We shall prove that  \[\boldsymbol{W}:= \bigcup_{x\in X}\bigcup_{n\geq n_x}E(x,n) \] is a wave.

\begin{lem}\label{l: stabilazing stuff}
For every $ x\in X $ and  $  \ell\geq m\geq n_x $, 
\begin{enumerate}
\item\label{item stabilize} $ I_m\cap  E(x,m)=I_{\ell}\cap E(x,m) $,
\item\label{item same circuit} $ C_M(y,I_\ell)=C_M(y,I_{m})\subseteq E(x,m)$ for every $y\in  E(x,m)\setminus I_m $, 
\item \label{item same cocircuit} $C_{N^{*}}(y,\overset{\circ}{\mathsf{span}}_{M}(I_\ell)^{1})=
C_{N^{*}}(y,\overset{\circ}{\mathsf{span}}_{M}(I_{m})^{1})  $ for every $y\in  E(x,m)\cap I_m^{1} $,
\item\label{item subdigraph} If $ yz\in  D(I_m)$ with $y,z\in E(x,m)$, then  $yz\in  D(I_\ell) $,
\item\label{item increasing}  $ E(x,m)\subseteq E(x,\ell) $.
\end{enumerate}
\end{lem}
\begin{proof}
Suppose that there is an $ n\geq n_x $ such that we know already the statement  whenever  $m,\ell \leq n $. For the induction step 
it is 
enough to 
show that the claim holds for $ n $ and $ n+1 $. We may assume that $ P_n $ exists, i.e. $ I_n\neq I_{n+1} $, since otherwise 
there is nothing to prove.

\begin{prop}\label{p: nx no aug}
 $ P_n \cap E(x,n)=\varnothing $.
\end{prop}
\begin{proof}
A common element of  $ P_n $ and $ E(x,n) $ would show that there is also an augmenting path in $ D(I_n) $ starting at $ x $ 
which is impossible since $ x\in X $ and  $ n\geq n_x $.
\end{proof}
\begin{cor}
$ I_{n}\cap  E(x,n)=(I_n \vartriangle P_n)\cap E(x,n) $.
\end{cor}
\begin{prop}
 $ (I_n \vartriangle P_n)\cap  E(x,n)=I_{n+1}\cap E(x,n) $.
\end{prop}
\begin{proof}
If $y\in  E(x,n)\setminus I_n  $, then its out-neighbours in $ D(I_n) $ are in $ E(x,n)\cap I_n $ and span $ y $ in $ M $. 
Thus  $I_{n+1}\setminus (I_n \vartriangle P_n)$ cannot contain any  edge from $ E(x,n) $.
\end{proof}
\begin{cor}\label{circ subset}
$ I_n\cap  E(x,n)=I_{n+1}\cap E(x,n) $  and  for every $y\in  E(x,n)\setminus I_n $ we have 
$ C_M(y,I_n)=C_M(y,I_{n+1})\subseteq E(x,n)$.
\end{cor}
\begin{cor}\label{cor: cocircuit stabil}
For $y\in  E(x,n)\setminus I_n $,  $ y $ has the same out-neighbours in $ D(I_n) $ and in $ D(I_{n+1}) $ and they span $ y $ in 
$ N^{*} $.  More concretely:
\[ C_{N^{*}}(y,\overset{\circ}{\mathsf{span}}_{M}(I_{n+1})^{1})=
C_{N^{*}}(y,\overset{\circ}{\mathsf{span}}_{M}(I_{n})^{1}). \]
\end{cor}

Finally,  for $ y\in  E(x,n)$, $ P_n $ does not contain $ y $ or any of its out-neighbours with respect 
to $ D(I_n) $  because $ P_n\cap E(x,n)=\varnothing  $. Hence by applying Lemma \ref{l: arc remain lemma} with $P_n,  y$ and 
$I_n $   (and then Observation \ref{arc remain fact}) we may conclude that  $ yz\in D(I_{n+1}) $ whenever   $yz\in  D(I_n) $.  
This implies $ E(x,n)\subseteq E(x,n+1) $  since reachability from $ x $ is witnessed by the same directed paths.
\end{proof}

Let \[ B:=\bigcup_{m\in \mathbb{N}} \bigcap_{n>m}W\cap I_n. \] We are going to show that $ B $ witnesses that $ W $ is a 
wave. Since $ M $ is
finitary the $ M $-independence of the sets  $ I_n\cap W $ implies the $ M $-independence of  $ B $. Similarly $ B^{0} $ is 
independent in $ N $ because $ N \upharpoonright E_0 $ is finitary.  Statements (\ref{item stabilize}) and 
(\ref{item same circuit}) of Claim \ref{l: stabilazing stuff} ensure $W\subseteq 
\mathsf{span}_{M}(B) $, while (\ref{item stabilize}) and (\ref{item same cocircuit}) guarantee $ B^{1} \subseteq 
\mathsf{span}_{N^{*}}(W\setminus B) $. The latter means that $ B^{1} $ is independent in $ N.(W^{1}\setminus B^{1}) $.  
Suppose for 
a contradiction that $ B^{0} $ is not independent in $ N.W^{0} $. Then there exists an $ N $-circuit $ C\subseteq E_0 $ that  
meets $ B $ 
but avoids $ W\setminus B $.  We already know that $ B^{0} $ is $ N $-independent thus $ C $ is not contained in $ B $. Hence
$C\setminus B= C\setminus W\neq\varnothing  $. Let us pick some $ e\in C\cap B$.
Since $ C $ is finite, for every large enough $ n $ we have $ C\cap B\subseteq C\cap I_n $ and  $ I_n $ spans  $ C $ in $ N $ (for 
the latter we use 
$ X\subseteq W\setminus B $). Applying Corollary \ref{cor: Noutgoing arc}
 with $I_n,  N, C $ and $ e $ tells that  $e\in C_N(f,I_n) $ 
for 
some $ f\in C\setminus W $ whenever $ n $ is large enough. 
Then by (\ref{item increasing}) of Claim \ref{l: stabilazing stuff} we can take an $ x\in X $ and an $ n\geq n_x $ such that $ e\in 
E(x,n) \cap C_N(f,I_n) $ 
for some $ f\in C\setminus W $. Then by definition $ f \in E(x,n) \subseteq W $ which contradicts $ f\in C\setminus W $. Thus  
$ B^{0} $ is indeed
independent in $ N.W^{0} $ and hence  $ B $ in $ N.W $ as well therefore $ W $ is a wave.
 
By $ \mathsf{cond}^{+}(M,N) $ we know that $ W $ consists of $ M $-loops and $ r(N.W)=0 $. It implies  $ r(N.X)=0 $ 
because $ X\subseteq W $ by definition. This means $ X\subseteq \mathsf{span}_N(E\setminus X) $. Since $X\subseteq E_0 $ 
and $ E_0 $ is the union of the finitary $ N $-components, $ X\subseteq \mathsf{span}_N(E_0\setminus X) $ follows. Thus for 
every $ x\in X $ 
there is a finite $ N $-circuit $ C\subseteq E_0 $ with $ C\cap X=\{ x \} $. The sequence $ (\mathsf{span}_N(I_n)\cap E_0 )$ is 
ascending by construction and exhausts $ E_0\setminus X $ by the definition of $ X $. As $ C-x\subseteq E_0\setminus X $ is  
finite, this implies that for every large enough $ n $, $ I_n $
spans $ C-x $ in $ N $  and hence spans $ x $ itself as well. But then by the definition of $ X $, we must have
$ X=\varnothing $. Therefore $ (\mathsf{span}_N(I_n)\cap E_0) $    exhausts $ E_0 $  and the proof of
Lemma \ref{l: key-lemma} is complete.
\end{proof}

 \section{An application: Degree-constrained orientations of infinite graphs}\label{s: application}
Matroid intersection is a powerful tool in graph theory and in combinatorial optimization. Our generalization 
Theorem \ref{t: main result} extends the scope of  its applicability to infinite 
graphs. To illustrate this, let us consider a classical problem in combinatorial optimization.  A graph is given with 
degree-constrains and we are looking for either an orientation 
that satisfies it or a certain substructure  witnessing the non-existence of such an orientation (see \cite{hakimi1965degrees}). 

Let  a  (possibly infinite) graph $ G=(V,E) $ be fixed through this section. We denote the set of edges incident with $ v $ by $ 
\boldsymbol{\delta(v)} $. 
Let $ o: V\rightarrow \mathbb{Z}$ with $ \left|o(n)\right|\leq 
d(v) $ for $ 
v\in V $ which we will threat as `lower bounds' for in-degrees in orientations in the following sense. We say that the orientation $ 
D $ of $ G $ is  
\emph{above} $ 
o $ at $ v $ if either $ 
o(v)\geq 0 $ and $ v $ has at least $ o(v) $ ingoing edges in $ D $ or  $ o(v)< 0 $ and all but at most $ -o(v) $ edges in $ \delta(v) 
$ 
are oriented towards $ v $ by $ D $. We say \emph{strictly above} if we forbid equality in the definition. Orientation $ D $ is 
above $ o $ if it is above $ o $ at every $ v\in V $. We say that $ D $ is (strictly) bellow $ o $ at $ v $ if the reverse of $ D $ is 
(strictly) above $ -o(v) $. Finally, $ D $ is (strictly) bellow $ o $ if the reverse of $ D $ is strictly above $ -o $.
 
\begin{thm}\label{t: indegree demand}
Let $ G=(V,E) $ be a countable graph and let $ o: V\rightarrow \mathbb{Z} $. If there is no 
orientation of $ G $ above $ o $, then there is a  $ V'\subseteq V $ and an orientation $ D $ of $ G $ such that
\begin{itemize}
\item $ D $ is bellow $ o $ at every $ v\in V' $;
\item There exists a $ v\in  V' $ such that $ D $ is strictly bellow $ o $ at $ v $; 
\item Every edge between $ V' $ and $ V\setminus V' $ is oriented by $ D $ towards $ V' $.
\end{itemize}
\end{thm}
\begin{proof}
 Without loss of generality we may assume that $ G $ is loopless. We define the digraph $ 
 \overset{\leftrightarrow}{G}=(V, A) $ by 
 replacing each $ e\in E 
 $ by back  and forth arcs $ a_{e}, a'_e $ between the 
 end-vertices of $ e $. Let $ \delta^{+}(v) $ be the set of the ingoing edges of $ v $ in  $ \overset{\leftrightarrow}{G} $.  For 
 $ v\in V $, let $ M_v $ be $ 
 U_{\delta^{+}(v), o(v)} $ if $ 
 o(v)\geq 0 $ and $ 
U_{\delta^{+}(v),-o(v)}^{*} $ if $ o(v)<0 $. We define $ N_e $ to be $ U_{\{a_e, a'_e  \}, 1} $ for $ e\in E $. Let \[ 
M:=\bigoplus_{v\in V}M_v\text{ and }N:=\bigoplus_{e\in E} N_e. \] Since $ M, N\in 
(\mathfrak{F}\oplus \mathfrak{F}^{*})(A) $,  Theorem \ref{t: main result} guarantees that there exists an $ I\in 
\mathcal{I}_M\cap 
\mathcal{I}_N $ and a partition $ A=A_M\sqcup A_N $  such that $I_M:=I\cap A_M $ spans $ A_M $ in $ M $ and  $ 
I_N:=I\cap A_N $ spans $ A_N $ in $ N $.  Note that  $ \left|I\cap \{ a_e, a'_e \} \right|\leq 1$ by the $ N $-independence of $ I 
$. We define $ D $ by taking the orientation $ a_e $ of $ e $ if $ a_e\in I $ and $ a'_e $ otherwise. Let $ V'' $ consists of those 
vertices $ v $ for which $ I_M $ contains a base of $ M_v $ and let $ V':=V\setminus V'' $.
We claim that whenever an edge $ e\in E $ is incident with some $ v\in V' $, then $ I $ contains one of $ a_e $ and $ a'_e $. 
Indeed, if $ I $ contains none of them then they cannot be $ N $-spanned by $ I_N $ thus they are $ M $-spanned by $ M $ which 
implies that both end-vertices of $ e $ belong to $ V'' $, contradiction. Thus if $ e $ is incident some $ v\in V' $, then all ingoing 
arcs of $ v $ in $ D $ must be in $ I $. Then the $ M $-independence of $ I $ ensures that  $ D  $ is bellow $ o $ at $ v $.

Suppose for a contradiction that $ a_e $ is an arc in $ D $ from a $ v\in V' $ to a $ w\in V'' $. As we have already shown, we must 
have $ a_e\in I $ . By $ w\in V'' 
$ we know that $ I_M $ contains a base of 
$ M_w $ thus $ a_e\notin I_N $ by the $ M $-independence of $ I $ and therefore $ a_e\in I_M $. But then $ a'_e $ cannot be 
spanned by $ I_N $ in $ N $ hence $ a'_e\in \mathsf{span}_M(I_M) $, which means that $ I_M $ contains a base of $ M_v $ 
contradicting $ v\in V' $. We conclude that all the edges between $ V'' $ and $ V' $ are oriented towards $ V' $ in $ D $. By the 
definition of $ V'' $,  $ D $ is above $ o $ at every $ w\in V'' $. If $ D $ is also above $ o $ for every $ v\in V' $, then $ D $ 
is above $ o $. Otherwise there exists a $v\in  V' $ such that $ D $ is strictly bellow $ o $ at $ v $, but then $ V' $ is as desired.
\end{proof}

Easy calculation shows that if $ G $ is finite, then  the existence of a $ V' $ described in Theorem \ref{t: indegree demand} 
implies the non-existence of an orientation above $ o $. Indeed, the total demand by $ o $ on $ V' $ is more than the number of all 
the edges that are incident with a vertex in $ V' $. That is why for finite $ G $, ``if'' can be replaced by ``if and only if'' in 
Theorem \ref{t: indegree demand}. For an infinite $ G $ it is not always the case. Indeed, let $ G $ be the one-way infinite path $ 
v_0, v_1,\dots $ and let $ o(v_n)=1 $ for $ n\in \mathbb{N} $. Then orienting edge $ \{ v_{n}, v_{n+1} \} $ towards $ v_{n+1} 
$ for each $ n\in \mathbb{N} $ and taking $ V':=V $ satisfies the three points in Theorem \ref{t: indegree demand}. However, 
taking the opposite orientation is above $ o $.

A natural next step is to  introduce upper bounds $ p:V\rightarrow \mathbb{Z}  $ beside the lower bounds 
$ o:V\rightarrow \mathbb{Z}  $. To avoid trivial obstructions we assume that $ o $ and $ p $ are \emph{consistent} which means 
that for every $ v\in V $ there is an orientation $ D_v $ 
which is above $ o $ at bellow $ p $ at $ v $.

\begin{quest}
Let $ G $ be a countable graph and let $ o, p: V\rightarrow \mathbb{Z} $ be a consistent pair of bounding functions.  
Suppose that there are orientations $ D_o $ and $ D_p $  that are above $ o $ and bellow $ p $ respectively. Is there always a 
single 
orientation $ D $ which is above $ o $ and bellow $ p $?
\end{quest}

The positive answer for finite graphs is not too hard to prove,  as far we know its first appearance in the literature is 
\cite{frank1978orient}.


\end{document}